\documentclass[12pt]{amsart}
\usepackage{amssymb,amsthm,amsmath}
\usepackage{lscape}
\usepackage[headings]{fullpage}
\usepackage{amsfonts}
\usepackage{paralist}
\usepackage{xspace}
\usepackage{euscript}
\usepackage{color}
\usepackage{epigraph}
\usepackage[all]{xy}
\usepackage[normalem]{ulem}
\usepackage[colorlinks,pagebackref=true,pdftex]{hyperref}
\usepackage[alphabetic,backrefs,abbrev,lite]{amsrefs} 

\makeatletter
\@addtoreset{equation}{section}
\def\theequation{\thesection.\@arabic \c@equation}

\def\theenumi{\@roman\c@enumi}
\def\@citecolor{blue}
\def\@linkcolor{blue}
\def\@urlcolor{blue}

\makeatother

\newtheorem{lemma}[equation]{Lemma}
\newtheorem{prop}[equation]{Proposition}
\newtheorem{cor}[equation]{Corollary}

\newtheorem{claim*}{Claim}
\newtheorem{thm}[equation]{Theorem}
\newtheorem{question}[equation]{Question}

\theoremstyle{definition}
\newtheorem{remark}[equation]{Remark}
\newenvironment{rmk}[1][]{\begin{remark}[#1] \pushQED{\qed}}{\popQED \end{remark}}
\newtheorem{eg}[equation]{Example}
\newenvironment{example}[1][]{\begin{eg}[#1] \pushQED{\qed}}{\popQED \end{eg}}
\newtheorem{definition}[equation]{Definition}
\newenvironment{defn}[1][]{\begin{definition}[#1]\pushQED{\qed}}{\popQED \end{definition}}
\newtheorem{notn}[equation]{Notation}

\def\<{\langle}
\def\>{\rangle}

\newcommand{\initial}{\operatorname{in}}

\title{Robust Toric Ideals}
\author{Adam Boocher and Elina Robeva}
\begin{document}

\maketitle
\begin{abstract}
We call an ideal in a polynomial ring robust if it can be minimally generated by a universal Gr\"obner basis.  In this paper we show that robust toric ideals generated by quadrics are essentially determinantal.  We then discuss two possible generalizations to higher degree, providing a tight classification for determinantal ideals, and a counterexample to a natural extension for Lawrence ideals.  We close with a discussion of robustness of higher Betti numbers.
\end{abstract}

\section{Introduction}
Let $S = k[x_1,\ldots,x_n]$ be a polynomial ring over a field $k$.  We call an ideal \emph{robust} if it can be minimally generated by a universal Gr\"obner basis, that is, a collection of polynomials which form a Gr\"obner basis with respect to all possible monomial term orders.  Robustness is a very strong condition.  For instance, if $I$ is robust then 
the number of minimal generators of each initial ideal is the same: 
\begin{equation}\label{equality in intro}
\mu(I) = \mu (\initial_< I)\ \mbox { for all term orders $<$.}\end{equation}

In general, we can only expect an inequality ($\leq$).

For trivial reasons, all monomial and principal ideals are robust.  Simple considerations show that robustness is preserved upon taking coordinate projections and joins (see Section 2). 
However, nontrivial examples of robust ideals are rare.   A difficult result of \cite{MR1229427,MR1212627} (recently extended by \cite{Boocher} and \cite{ConcaGorlaDeNegri}) shows that the ideal of maximal minors of a generic matrix of indeterminates is robust.  In the toric case, the class of Lawrence ideals (which naturally includes determinantal ideals) is the largest known class of robust ideals.  

This paper discusses robustness for toric ideals.  Our main result is the following:

\begin{thm}\label{IntroThm}
Let $F$ be a set of irreducible binomials that minimally generate an ideal.  (Assume that $F$ cannot be partitioned into disjoint sets of polynomials in distinct variables.)  If $F$ consists of polynomials of degree $2$ then the following are equivalent:
\begin{itemize}
\item The ideal generated by $F$ is robust
\item $|F| = 1$ or $F$ consists of the $2\times 2$ minors of a generic $2\times n$ matrix
$$\begin{pmatrix}x_1 & \cdots & x_n \\ y_1 & \cdots & y_n \end{pmatrix}$$
up to a rescaling of the variables.
\end{itemize}
In higher degree, it is not true that the robustness of $(F)$ implies that it is determinantal (or even Lawrence).
\end{thm}

Our motivation for studying robust toric ideals is threefold.  First, the study of ideals minimally generated by a Gr\"obner basis (for some term order) is ubiquitous in the literature.  In \cite{ConcaHostenEtal}, Conca et al studied certain classical ideals and determined when they are minimally generated by some Gr\"obner basis.   In the study of Koszul algebras, one of the most fruitful approaches has been via $G$-quadratic ideals - those generated by a quadratic Gr\"obner basis.  We are not aware, however, of any systematic study of ideals minimally generated by a universal Gr\"obner basis; \emph{robust ideals}.  

Second, we focus on toric ideals, because they are a natural testing ground for algebraic questions.  Toric ideals possess binomial universal Gr\"obner bases, and hence and there is a rich theory of Gr\"obner basis theory, see e.g. \cite{ConvexPoly}.  

Finally, passing from an ideal to an initial ideal is a particular type of flat deformation.  In this phrasing, robustness is almost equivalent to the property that the minimal number of generators is preserved by these deformations.\footnote{Note that the equality \ref{equality in intro} is in general, weaker than robustness.  For example, consider the ideal $(x-y,y-z)$.}  This interpretation suggests there might be a geometric interpretation of robustness in terms of the Hilbert scheme.

The paper is organized as follows: In section 2, we prove our main result, Theorem \ref{mainthm} characterizing robust toric ideals generated in degree two. The methods are mainly combinatorial. In Sections 3 and 4 we pose two questions concerning extensions of Theorem \ref{mainthm} using Lawrence ideals. We provide negative and positive answers respectively.  Section 5 closes with a discussion of ``robustness of higher Betti numbers,'' our original motivation for this project.

\section{Quadratic Robust Toric Ideals are Determinantal}
\noindent In the sequel, by \emph{toric ideal} we will always mean a prime ideal generated minimally by homogeneous binomials with nonzero coefficients in $k$.  By the support of a polynomial we mean the set of variables appearing in its terms.  

\begin{defn}
A set $F$ of polynomials in $S$ is called robust if $F$ is a universal Gr\"obner basis and the elements of $F$ minimally generate their ideal.  
\end{defn}

If a set of polynomials $F$ can be written as a union $F = G \cup H$ of polynomials in disjoint sets of variables, then we say that $G$ is a \emph{robust component} of $F$.  If $F$ admits no such decomposition, then we say the set $F$ is \emph{irreducibly robust}.  Notice that robustness is preserved under these disjoint unions, so to classify robust ideals, it suffices to study the irreducible ones.    We remark that the ideal of $F$ corresponds to the join of the varieties corresponding to the $G$ and $H$.

The goal of this section is to prove the following theorem:

\begin{thm}\label{mainthm}
Let $F$ be an irreducibly robust set consisting of irreducible quadratic binomials.   Then $F$ is robust if and only if $|F| = 1$ or $F$ consists of the $2\times 2$ minors of a generic $2\times n$ matrix
$$\begin{pmatrix}x_1 & \cdots & x_n \\ y_1 & \cdots & y_n \end{pmatrix}$$
up to a rescaling of the variables.
\end{thm}

\begin{rmk}
In the statement of the Theorem, we only assume that the generators are irreducible. It turns out that this is sufficient to show that the ideal they generate is prime. 
\end{rmk}

Notice that one direction follows immediately from the results of \cite{MR1212627} which show that the $2\times 2$ minors are a universal Gr\"obner basis.  To prove the converse, our technique is essentially to eliminate certain combinations of monomials from appearing in $F$.  To simplify notation, we will omit writing coefficients in the proofs when it is clear that they do not affect the argument.  In particular, we treat the issue of coefficients only in tackling the proof of Theorem \ref{mainthm} itself and not in earlier lemmas.  

\begin{lemma}\label{unique terms}
Let $F$ be a robust set of prime quadratic binomials.  Then no monomial appears as a term in two different polynomials in $F$.
\end{lemma}
\begin{proof}
Suppose that the monomial $m$ appears in the polynomials $f,g\in F$.  Let $<$ be a Lex term order taking the support of $m$ to be first.  
Since $f$ and $g$ are prime, $<$ will select $m$ as the lead term of both $f$ and $g$, and applying Buchberger's algorithm, we would obtain a degree zero syzygy of the elements in $F$, contradicting the minimality of $F$.
\end{proof}

\begin{prop}\label{restrict}
If $F$ is robust, and $0\leq k \leq n$, then so is $F\cap k[x_1,\ldots, x_k]$. 
\end{prop}
\begin{proof}
Write $F_k = F\cap k[x_1,\ldots,x_k]$. It is clear that $F_k$ minimally generates the ideal $(F_k)$.  Let $<$ be any term order on $k[x_1,\ldots,x_k]$.  Extend $<$ to a term order $<_S$ on $S$, taking $x_1,\ldots,x_k$ last.  Then since $F$ is a Gr\"obner basis with respect to $<_S$, by basic properties of Gr\"obner bases, we know that $F_k$ will be a Gr\"obner basis with respect to $<$.
\end{proof}

The above proposition is extremely useful, because in our analysis it will be helpful to assume we are working in a ring with few variables.  We will use this reduction extensively in the following main technical lemma.  We use the letters $a,\ldots,z$ when convenient for ease of reading.

\begin{lemma}\label{main lemma}
Let $F$ be a robust set of prime quadratic binomials:
\begin{enumerate}[(a)]
\item $F$ cannot contain two polynomials of the form
$$f = x^2 + yz, \ g = xy + m$$
or 
$$f = x^2 + y^2, \ g = xy + m$$
or 
$$f = x^2 + y^2, \ g = xz + m.$$
where $m$ is any monomial.
\item $F$ cannot contain two polynomials of the form
$$f=x_ix_j + x_kx_l, \ g=x_ix_k + x_px_q$$
(here we do not assume $i,j,k,l,p,q$ are distinct.)
\item $F$ cannot contain two polynomials of the form
$$f = x^2 + yz, \ g = xw + m$$
or 
$$f = x^2 + yz, \ g = yw + m$$
where $m$ is any monomial.
\item If $f,g\in F$ are two polynomials whose supports share a variable, then all terms of $f$ and $g$ are squarefree.
\item If $F$ contains two polynomials whose supports share a variable, then (up to coefficients) $F$ must contain the $2\times 2$ minors of a generic matrix.
$$\begin{pmatrix} a & b & c \\ d & e & f\end{pmatrix}$$
\end{enumerate}
\end{lemma}
\begin{proof}
a) We prove the first statement.  The proofs of the others are similar.  Suppose that $f,g\in F$.  Notice that by primality, $m$ cannot contain a factor of $y$.  Let $<$ be the lex term order with ($y>x>z>$ all other variables).  Then the $S$-pair of $f$ and $g$ is $x^3 - mz$, whose lead term is $x^3$.  Since $F$ is a Gr\"obner basis with respect to $<$ we must have some polynomial $h$ whose lead term divides $x^3$.  Since $x^2$ is not the lead term of $f$, $h\neq f$ and we must have two distinct polynomials in $F$ with $x^2$ appearing.  This contradicts Lemma \ref{unique terms}.

b) Suppose that $f,g\in F$.  By restricting to the subring $k[x_i,x_j,x_k,x_p,x_q]$, Proposition \ref{restrict} tells us we can assume $F$ involves only these variables. Notice by primality (and part (a)) we know that $k$ and $i$ are distinct from $j,l,p,q$.  Let $<$ be the lex term order with $(x_k>x_i>$ all other variables.)  Then the $S$-pair of $f$ and $g$ is $x_i^2x_j - x_lx_px_q$, whose lead term is $x_i^2x_j$.  As in part a) we must have some polynomial $h\in F$ whose lead term divides $x_i^2x_j$.  The only possible monomials are $x_i^2$ and $x_ix_j$.  And since $x_ix_j$ appears in $f$ (and is not a lead term) we must have a polynomial $h = x_i^2 + x_ax_b \in F$ for some $a,b\in \{i,j,k,l,p,q\}$.  So $F$ contains
$$f = x_ix_j + x_kx_l, \ g = x_ix_k + x_px_q, \ h= x_i^2 + x_ax_b.$$
Applying part a), and primality, we know that $a,b \in \{l,p,q\}$.  By part a), we know that $x_ax_b$ must be squarefree, and since $x_px_q$ already appears, we can say (renaming $p$ and $q$ if necessary,) that $x_ax_b = x_lx_p$.  But now choosing $<$ to be the lex term order with ($x_l > x_p > x_i > x_k> $ all other variables,) we see that the $S$-pair of $f$ and $h$ is $x_i^2x_k - x_ix_jx_k$ whose lead term is $x_i^2x_k$ which is only divisible by the monomials $x_i^2$ and $x_ix_k$, neither of which can be a lead term of a polynomial in $F$ by Lemma \ref{unique terms}.

c) We will prove the first statement.  The second proof is similar.  Suppose that $f,g \in F$.  First restrict, using Proposition \ref{restrict} to assume we are working only with the variables $x,y,z,w$ and the factors of $m$.  Let $<$ be the lex term order with ($w>x>$ all other variables. Taking the $S$-pair of $f$ and $g$, we obtain $wyz - mx$, whose lead term is $wyz$.  Since this must be divisible by the lead term of some polynomial $h\in F$, without loss of generality, we assume $h = wy + n$.  Consider now the possibilities for $n$.  By primeness $n$ cannot contain a factor of $w$ or $y$.  By part b) it cannot contain a factor of $x$ or $z$.  Hence, the only possible options left are that the factors of $n$ are contained in the factors of $m$.  But this means that $m,n$ are $ab,a^2$ for some (new, distinct) variables $a,b$.  As in (b), we can conclude this is impossible.

d) This follows immediately from parts a) - c).

e) We assume that $F$ contains two polynomials whose supports intersect.  By d) we can assume that these polynomials are squarefree, and we write them as $p = ae - bd,q = af - m_1m_2$ where $m_i$ represents some variable. Notice that by primality and part b), neither $m_1$ nor $m_2$ can be $a,e$ or $f$.  Nor can $m_1m_2=bd$ (since it would be a repetition).  Hence we may as well assume $m_1$ is different from the other variables, and call it $c$.  There are now two cases:  Either $m_2$ is also a new variable $g$, or it isn't, in which case we can see that without loss of generality, $m_2 = d$.  Rewriting:  If $p,q$ are two polynomials whose supports intersect, then they must contain either $6$ or $7$ distinct variables.  

In the case of $6$ variables, restrict $F$ to the subring $k[a,b,c,d,e,f]$.  Now 
$$p = ae - bd, \ q= af - cd.$$  Computing an $S$-pair with the lex order ($a>b>d>$ all other variables) we obtain $fbd - ecd$ with lead term $fbd$.  This must be divisible by the lead term of some polynomial $r\in F$.  But this lead term cannot be $bd$ (by its presence in $p$), hence it must be either $bf$ or $df$.  In case it is $bf$ then $F$ contains a polynomial of the form $r = bf - n_1n_2$.  Now $n_1,n_2\in \{a,c,d,e\}$ by primality.  And part b) of this lemma allows us to further say $n_1,n_2\in \{c,e\}$ which along with squarefreeness implies that $r = bf - ce$ as required.  In case the term is $df$, similar considerations show that parts a) - d) will not allow any $n_1n_2$.

In the case of $7$ variables, restrict $F$ to the subring $k[a,b,c,d,e,f,g]$.  Now 
$$p = ae - bd, \ q = af - cg.$$
Computing an $S$-pair with the lex term order ($a>b>d>$ all other variables):  we obtain $fbd-ebg$ with lead term $fbd$.  This must be divisible by the lead term of some polynomial $r\in F$.  But this lead term cannot be $bd$ (by its presence in $p$), hence it must be either $bf$ or $df$.  By symmetry we can assume that it is $df$ and that $F$ contains a polynomial of the form $df - n_1n_2$.  Now $n_1,n_2\in \{a,b,c,e,g\}$ by primality, and part b) restricts us further to $n_1,n_2 \in \{c,e,g\}$.  And since $cg$ already appears, we can conclude that $n_1n_2 = eg$ or $ec$ (and again by symmetry, we may assume $n_1n_2 = ec$.  But now notice that $q$ and $r$ are two polynomials whose supports intersect, and involve only $6$ variables.  Hence, by the previous part of this proof, we can conclude that $F$ contains a polynomial $s = gd - ae$.  But this is a contradiction by Lemma \ref{unique terms}. 
\end{proof}

\begin{proof}[Proof of Theorem \ref{mainthm}]
Suppose that $|F|>1$.  Since $F$ is irreducible, it must contain two polynomials whose supports intersect.  By Lemma \ref{main lemma} we can conclude that $F$ contains polynomials of the form:
$$p_1 =  ae - bd, \ p_2 = af - cd, \ p_3 = bf - ce$$
up to coefficients.  However, computing an $S$-pair with the lex order on ($a>b>\cdots >f$) we obtain:
$$S(p_1,p_2) = bdf - cde \mbox{ (with some nonzero coefficients)}$$
which after reducing by $p_3$ we obtain either zero, or a constant multiple of $cde$.  In the latter case, in order to continue the algorithm, we would have to have another polynomial in $F$ whose lead term divided $cde$.   By the presence of $p_1,p_2,p_3$, the terms $cd$ and $ce$ are prohibited.  And by Lemma \ref{main lemma} b), $de$ is also prohibited.  Hence, this $S$-pair must reduce to zero after only two subtractions.

This means in fact, that the polynomials are precisely determinants of some matrix 
$$\begin{pmatrix} \lambda_1 a & \lambda_2 b &\lambda_3 c 
\\
\mu_1 d & \mu_2 e & \mu_3 f \end{pmatrix}$$
for some nonzero constants $\lambda_i, \mu_i$.

To complete the proof, suppose that $F\neq \{p_1,p_2,p_3\}$.  Since $F$ is irreducible, one polynomial $p_4\in F$  must share a variable with say, $p_1$.  Renaming variables if necessary, say that variable is $a$.  Then (ignoring constants for the moment) by applying the proof of Lemma \ref{main lemma} e), to the polynomials $p_1 = ae - bd$ and $p_4 = ah - m_1m_2$ we can conclude that $F$ contains the minors of the matrix 
$$\begin{pmatrix} \lambda_1 a & \lambda_2 b &\lambda_4 g 
\\
\mu_1 d & \mu_2 e & \mu_4 h \end{pmatrix}.$$

Applying this technique to $p_2$ and $p_4$ as well, shows that we in fact get all $2\times 2$ minors of the full matrix 

$$\begin{pmatrix} \lambda_1 a & \lambda_2 b & \lambda_3 c& \lambda_4 g 
\\
\mu_1 d & \mu_2 e &\mu_3 f &  \mu_4 h \end{pmatrix}.$$

Inductively we continue this process until we obtain all of $F$.  
\end{proof}
Notice that in our proof, every term order we used was a Lex term order, and we ended up a with a prime ideal.  Hence, we have the following: 
\begin{cor}
If $F$ is a set of prime quadratic binomials that minimally generate an ideal.  Then the following are equivalent:
\begin{enumerate}  \item $F$ is a Gr\"obner basis with respect to every Lex term order.
\item $F$ is a Gr\"obner basis with respect to every term order.
\item $F$ generates a prime ideal and the irreducible robust components of $F$ are generic determinantal ideals and hypersurfaces.
\end{enumerate}
\end{cor}

\begin{cor}
If $X$ is a generic $k\times n$ matrix, and $F$ is the set of $2\times 2$ minors, then $F$ is a universal Gr\"obner basis if and only if $k = 2$.
\end{cor}
\begin{remark}
It is almost the case that \emph{every} irreducibly robust component is determinantal.  Indeed, every prime binomial is up to rescaling either $xy-zw$ or $x^2-yz$.  The former is determinantal.  Thus the only possible non-determinantal robust component is $\{x^2-yz\}$.
\end{remark}

\section{From Determinants to Lawrence Ideals}
Encouraged by the result of the previous section, it is natural to ask to what extent robustness classifies generic determinantal ideals.  Indeed, it is easy to see that the ideal of minors of any $2\times n$ matrix whose entries are relatively prime monomials will be robust.  There are two questions we consider:

\begin{question}\label{question3.1} \ 
\begin{enumerate}[1.]
	\item \label{(i)} If $I$ is a robust toric ideal, is $I$ generated by the $2\times 2$ minors of some matrix of monomials?
	\item \label{(ii)}Precisely which matrices of monomials provide robust ideals of $2\times 2$ minors?
\end{enumerate}
\end{question}

The answer to the first question is negative.  Examples are provided by \emph{Lawrence ideals}, studied in \cite{ConvexPoly}.  If $I$ is any toric ideal, with corresponding  variety $X\subset \mathbb{P}^{n-1}$, then the ideal $J$ corresponding to the re-embedding of $X$ in $(\mathbb{P}^1)^{n-1}$ is called the Lawrence lifting of $I$. Its ideal is generated by polynomials of the following form:
$$J_L = \left( \mathbf{x^ay^b} -\mathbf{x^by^a} \ | \ \mathbf{a - b} \in L \right) \ \subset \ S = k[x_1,\ldots,x_n,y_1,\ldots, y_n],$$
where $L$ is a sublattice of $\mathbb{Z}^n$ and $k$ is a field.  Here $\mathbf{a} = x_1^{a_1}x_2^{a_2}\cdots x_n^{a_n}$ for $\mathbf{a} = (a_1,\ldots, a_n) \in \mathbb{N}^n$.  Binomial ideals of the form $J_L$ are called \emph{Lawrence ideals}.  The following result is Theorem 7.1 in \cite{ConvexPoly}.

\begin{prop}
The following sets of binomials in a Lawrence ideal $J_L$ coincide:
\begin{enumerate}[a)]
	\item Any minimal set of binomial generators of $J_L$.  
	\item Any reduced Gr\"obner basis for $J_L$.  
	\item  The universal Gr\"obner basis for $J_L$.  
	\item The Graver basis for $J_L$.  
\end{enumerate}
\end{prop}

Hence Lawrence ideals provide a large source of robust toric ideals, and naturally include the class of generic determinantal ideals.  Given this, it is natural to rephrase the first part of Question \ref{question3.1} as: \begin{question}
Does robustness characterize Lawrence ideals?  
\end{question}

Again the answer is negative.

\begin{example}\label{counterexample} The ideal 
$$\mathtt{I = (b^2e-a^2f,bc^2-adf,ac^2-bde,c^4-d^2ef)}$$
in the polynomial ring $\mathbb{Q}[a,b,c,d,e,f]$ is robust but not Lawrence.  This example was found using the software {\tt Macaulay2} and {\tt Gfan} \cite{M2,gfan}. It is the toric ideal $I_L$ corresponding to the lattice defined by the kernel of 
$$L = \begin{pmatrix} 
	1&1&1&1&1&1 \\ 
	1&1&0&0&0&0 \\
	0&0&1&2&0&0 \\
	1&0&1&0&3&1
	\end{pmatrix}
	$$\end{example}
	
Given this counterexample we ask
\begin{question}
Is there a nice combinatorial description of robust toric ideals?
\end{question}
\begin{rmk}
Heuristically, it is very easy to find robust ideals that are not Lawrence by starting with a Lawrence ideal given by $J_L$.  This lattice $L$ gives rise to a lattice $\tilde{L}\subset \mathbb{N}^{2n}$ such that $J_L = I_{\tilde{L}}$.  By modifying $\tilde{L}$ slightly, it is very often the case that the resulting toric ideal is robust (though often non-homogeneous).  The ubiquity of these examples computationally suggests that a nice combinatorial description of robustness may require imposing further hypotheses.
\end{rmk}

\section{Matrices of Monomials}

\noindent In this section we answer Question \ref{question3.1}.\ref{(ii)}.

\begin{thm} Suppose that $X_i, Y_j$ are monomials of degree at least 1 in some given set of variables $\mathcal{U} = \{ u_1, u_2, \ldots,
  u_d \}$. Let 
  \begin{eqnarray*}
    A = \left(\begin{array}{cccc}
      X_1 & X_2 & \cdots & X_n\\
      Y_1 & Y_2 & \cdots & Y_n
    \end{array}\right), &  & 
  \end{eqnarray*}
where $n \geqslant 3$ and suppose that the set $F$ of $2 \times 2$-minors $X_i Y_j - X_j Y_i$, $i \neq j$ consists of irreducible binomials. Then $F$ is robust if and only if all the monomials $X_i, Y_j$ are relatively prime.  
\end{thm}
The proof is technical, so we begin by fixing notation.  
Since we will assume that each $X_i Y_j -
  X_j Y_i$ is prime for all $i \neq j$, then, $\gcd (_{} X_i, X_j) = \gcd (
  Y_i, Y_j) = \gcd ( X_i, Y_i) = 1$ for all $i \neq j$. Therefore, if we define
  \begin{eqnarray*}
    z_{i j} = \gcd ( X_i, Y_j), &  & 
  \end{eqnarray*}
  then, we can write
  \begin{eqnarray*}
    X_i = x_i  \prod_{j \neq i} z_{i j} & \text{ and } & Y_j = y_j \prod_{i \neq
    j} z_{i j} .
  \end{eqnarray*}
  Thus,
\[
    \gcd ( x_i, x_j) = \gcd ( y_i, y_j) = 1 \text{ for } \text{all } i \neq j
    \tag{i}
\]
\[
    \gcd ( z_{i j}, z_{k l}) = 1 \text{ whenever } i \neq k \text{ or } j \neq l, 
    \tag{ii}
\]
\[
    \gcd ( x_i, z_{k l}) = 1 \text{ if } i \neq k  \text{ and } \gcd ( y_j,
    z_{k l}) = 1 \text{ if } j \neq l.  \tag{iii}
\]
  Our goal is to show that $z_{i j} = 1$ for all $i \neq j$.
  
\begin{lemma}\label{lemma41}
If $z_{12} \neq 1$ then $n\geqslant 4$ and for each $m\geqslant 3$, there exist permutations $i_m,l_m \neq 1,2,m$ and $j_m,k_m\neq 1,2$ and term orders $>_1$ and $>_2$ such that 
\[
    X_{i_m} Y_{j_m} \mid X_2 Y_m^2  \frac{X_1}{z_{1 2}} \text{ and }  X_{i_m} Y_{j_m} >_1 X_{j_m}
    Y_{i_m}  \tag{iv}
\]
\[
    X_{k_m} Y_{l_m} \mid X_m^2 Y_1  \frac{Y_2}{z_{1 2}} \text{ and }  X_{k_m}Y_{l_m} >_2 X_{l_m}
    Y_{k_m}  \tag{iv'}.
\]
Moreover, 
\[X_{i_m} = x_{i_m}z_{i_m m} \text{ and } x_{i_m} | z_{i_m m},
\]
\[Y_{l_m} = y_{l_m}z_{m l_m} \text{ and } y_{l_m} | z_{m l_m}.
\]
\end{lemma}
\begin{proof}
We will build $<_1$ and $<_2$ in several steps. To begin, take a lex term order $>$ where the variables in $z_{1 2}$ are first.
  Consider the $S$-pair:
  \begin{eqnarray*}
    S \left( \underline{X_1 Y_m} - X_m Y_1, \underline{X_m Y_2} - X_2 Y_m
    \right) = &  & 
  \end{eqnarray*}
  \begin{eqnarray*}
     = \frac{\text{lcm} ( X_1 Y_m , X_m Y_2)}{X_1 Y_m} (
    X_1 Y_m - X_m Y_1) - \frac{\text{lcm} ( X_1 Y_m , X_m Y_2)}{X_m Y_2} ( X_m Y_2 - X_2 Y_m) =
  \end{eqnarray*}
  \begin{eqnarray*}
    = X_m \frac{Y_2}{z_{1 2}} ( X_1 Y_m - X_m Y_1) - Y_m \frac{X_1}{z_{1 2}} (
    X_m Y_2 - X_2 Y_m) = &  & 
  \end{eqnarray*}
  \begin{eqnarray*}
    = X_2 Y_m^2  \frac{X_1}{z_{1 2}} - X_m^2 Y_1 \frac{Y_2}{z_{1 2}} . &  & 
  \end{eqnarray*}
  Since all of the variables in $X_2 Y_m^2  \frac{X_1}{z_{1 2}} - X_m^2 Y_1 
  \frac{Y_2}{z_{1 2}}$ are different from the variables in $z_{1 2}$, then, there exist term orders $>_1$ and $>_2$ refining $>$ for which $X_2 Y_m^2  \frac{X_1}{z_{1 2}}$ is the
  leading term for $>_1$ and $X_m^2 Y_1  \frac{Y_2}{z_{1 2}}$
  is the leading term for $>_2$.  \\
  
   Consider first $>_1$: $X_2 Y_m^2 
  \frac{X_1}{z_{1 2}} >_1 X_m^2 Y_1  \frac{Y_2}{z_{1 2}}$. Since the $X_i
  Y_j - X_j Y_i$ form a Gr\"obner basis with respect to $>_1$, there exist $i_m \neq j_m$ such that
\[
    X_{i_m} Y_{j_m} \mid X_2 Y_m^2  \frac{X_1}{z_{1 2}} \text{ and }  X_{i_m} Y_{j_m} >_1 X_{j_m}
    Y_{i_m}  \tag{iv}
\]
  in this ordering. If $j_m = 2$, then,  $z_{1 2} \mid Y_2 \mid
  X_2 Y_m^2  \frac{X_1}{z_{1 2}}$, which is not true. If $j_m = 1$, then, since
  $z_{1 2} \mid X_1 \mid X_1 Y_{i_m}$ and $z_{1 2} \nmid X_{i_m} Y_1$ and $z_{1 2}$ was chosen to
  have its variables first in $>_1$, then, $X_{j_m} Y_{i_m} = X_1 Y_{i_m}
  >_1 X_{i_m} Y_1 = X_{i_m} Y_{j_m}$, which is not true by (iv). Thus, $j_m \geqslant 3$.
Similarly, we can deduce that $i_m\geqslant 3$.
Thus, $i_m, j_m \geqslant 3$. 

Since $i_m \geqslant 3$, $\gcd ( X_1, X_{i_m}
  ) = \gcd ( X_2, X_{i_m}) = 1$, and we are assuming that $X_i,Y_i\neq 1$ for all $i$ and $X_{i_m} \mid
  X_2 Y_m^2  \frac{X_1}{z_{1 2}}$ then, $X_{i_m} \mid Y_m^2$. Since $\gcd ( X_m,
  Y_m) = 1$, $i_m \neq m$ either. Thus, we have that $i_m \neq 1,2,m \text{ and } , j_m
  \geqslant 3$. So, in particular, if $n = 3$, we already have a
  contradiction. We assume now that $n \geqslant 4$.  Since $X_{i_m} = x_{i_m} \prod_{j \neq i_m} z_{i_m j}$ and $Y_m = y_m \prod_{k \neq m}
  z_{k m}$, properties (i) and (iii) allow us to conclude that $X_{i_m}
  = x_{i_m} z_{i_m m}$ and $x_{i_m} \mid z_{i_m m}$.

Going back to when we chose the ordering $>_1$, consider now the
  ordering $>_2$ for which $X_m^2 Y_1  \frac{Y_2}{z_{1 2}}
  >_2 X_2 Y_m^2  \frac{X_1}{z_{1 2}}$. By a symmetric argument we
  find that there exist $k_m \geqslant 3, l_m\neq 1,2,m$ such that $X_{k_m} Y_{l_m} \mid
  X_m^2 Y_1  \frac{Y_2}{z_{1 2}}$ and, thus, $Y_{l_m} = y_{l_m} z_{m l_m}$ with $y_{l_m}
  \mid z_{m l_m}$.\\
  
    Hence, there exist $i_m, l_m \neq 1,2,m$ and $j_m,k_m\neq 1,2$ such that $X_{i_m} = x_{i_m} z_{i_m m}$, $x_{i_m} \mid
  z_{i_m m}$ and $Y_{l_m} = y_{l_m} z_{m l_m}$, $y_{l_m} \mid z_{m l_m}$.
\end{proof}

\begin{lemma} If $z_{1 2}\neq 1$, then, $n$ is even and we can rearrange the numbers ${1,..,n}$ so that for each $i \leqslant \frac{n}{2}$,
\[
X_{2i} = x_{2i} z_{2i, 2i+1} \text{ and } Y_{2i} = y_{2i} z_{2i+1, 2i}
\]
\[
X_{2i+1} = x_{2i+1} z_{2i+1, 2i} \text{ and } Y_{2i+1} = y_{2i+1} z_{2i, 2i+1}
\]
and $x_{2i}, y_{2i+1} \mid z_{2i, 2i+1}$ and $x_{2i+1}, y_{2i} \mid z_{2i+1, 2i}$.
\end{lemma}
\begin{proof}
By property (ii) and Lemma 4.2 we have that $m \mapsto i_m$ and $m \mapsto l_m$ are permutations on $\{3,..,n\}$ with no fixed points. Thus, for each $i\geqslant 3$, there exists $m\geqslant 3$ such that $m\neq i$ and $i = i_m$, $X_i = x_i z_{i m}$ with $x_i\mid z_{i m}$ and for each $l\geqslant 3$, there exists $m\geqslant 3$, $m\neq l$ such that $l = l_m$ and $Y_l = y_l z_{m l}$.\\

Fix $m'\neq 1,2$. Then, $X_{i_{m'}} = x_{i_{m'}} z_{i_{m'} m'}$ and $x_{i_{m'}} \mid z_{i_{m'} m'}$. Since we
  assumed that $X_{i_{m'}} \neq 1$, then, $z_{i_{m' }m'} \neq 1$. So now, repeating
  the whole argument with $m'$ and $i_{m'}$ instead of with $1$ and $2$ (recall
  that $i_{m'} \neq 1, 2, m'$), we would get similar permutations on $\{ 1,
  \ldots, n \} \setminus \{ m', i_{m'} \}$. But, by property (ii), these permutations have to agree with the permutations $m \mapsto
  i_m$ and $m \mapsto l_m$ from above on the set $\{ 1, \ldots, n \} \setminus
  \{ 1, 2, m', i_{m'} \}$. Thus, $( 1, 2)$ and $( m', i_{m'})$ will be transpositions
  in all of these permutations (and, in particular, $i_{m'} = l_{m'}$). \\
  
  Since we can run the above argument with any $m'\neq 1,2$, we have that the permutations $m\mapsto i_m$ and $m\mapsto l_m$ agree and are composed of transpositions $(m, i_m)$. In particular, $n$ is even and, after rearranging the numbers from $\{1,..,n\}$ so that $i_{2k} = 2k+1$ and, thus, $i_{2k+1} = 2k$ for all $k =1,..,n/2$, our matrix $A$ looks as follows:
   \begin{eqnarray*}
    A = \left(\begin{array}{ccccc}
      x_1 z_{1 2} & x_2 z_{2 1} & x_3 z_{3 4} & x_4 z_{4 3} & \cdots\\
      y_1 z_{2 1} & y_2 z_{1 2} & y_3 z_{4 3} & y_4 z_{3 4} & \cdots
    \end{array}\right) . &  & 
  \end{eqnarray*}
 The rest of the statement of the Lemma follows from Lemma 4.2.
\end{proof}

\begin{proof}[Proof of Theorem:] 

($\Rightarrow$): It suffices to show that $z_{ij} =1$ for all $i,j$.  Without loss of generality, assume that $z_{12} \neq 1$.  

By Lemma 4.2 we have that 
\[
  X_{i_m} Y_{j_m} \mid X_2 Y_m^2  \frac{X_1}{z_{1 2}}\ \ \text{ and }\ \ X_{k_m} Y_{l_m} \mid X_m^2 Y_1  \frac{Y_2}{z_{1 2}}
\]
for all $m\geqslant 3$. But by (the proof of) Lemma 4.3 we know that the above hold when we substitute $1$ and $2$ with any $m'$ and $i_{m'}$ such that $m\neq i_{m'},m'$, i.e.
\[
  X_{i_m} Y_{j_m} \mid X_{i_{m'}} Y_m^2  \frac{X_{m'}}{z_{m' i_{m'}}}\ \ \text{ and }\ \ X_{k_m} Y_{l_m} \mid X_m^2 Y_{m'}  \frac{Y_{i_{m'}}}{z_{m' i_{m'}}}
\]
Rewriting out the expressions in the form $X_{i_m} = x_{i_m} z_{i_m m}$ and $Y_{l_m} = y_{l_m} z_{m l_m}$ and then canceling repeating factors, we get that
\[
  x_{i_m} y_{j_m} z_{l_{j_m} j_m } \mid x_{i_{m'}} z_{i_{m'} m'} y_m^2 z_{i_m m} x_{m'}\ \ \text{ and }\ \ x_{k_m} z_{k_m l_{k_m}} y_{l_m} \mid x_m^2 z_{m l_m} y_{m'} z_{i_{m'} m'}  y_{i_{m'}}
\]
for all $m'\neq m, i_m$. Therefore, we have that for every $m'\neq m, i_m$
\[
  z_{ l_{j_m} j_m} \mid x_{i_{m'}} z_{i_{m'} m'} y_m^2 z_{i_m m} x_{m'}\ \ \text{ and }\ \  z_{k_m l_{k_m}}  \mid x_m^2 z_{m l_m} y_{m'} z_{i_{m'} m'}  y_{i_{m'}}
\]
Noting that $x_{i_{m'}}\mid z_{i_{m'} m'}$ and $x_{m'}\mid z_{m' l_{m'}}$ and, similarly for $y_{m'}, y_{i_{m'}}$, switching $m'$ with $i_{m'}$, and using (ii), shows us that
\[
  z_{l_{j_m} j_m} \mid  y_m^2 z_{i_m m} \ \ \text{ and }\ \  z_{k_m l_{k_m}}  \mid x_m^2 z_{m l_m} 
\]
Again, by property (ii), the only way for this to happen is if $j_m = k_m = m$. In that case, we have that
\[
  x_{i_m} y_{m} z_{l_{m} m } \mid x_{i_{m'}} z_{i_{m'} m'} y_m^2 z_{i_m m} x_{m'}\ \ \text{ and }\ \ x_{m} z_{m l_{m}} y_{l_m} \mid x_m^2 z_{m l_m} y_{m'} z_{i_{m'} m'}  y_{i_{m'}}
\]
Again, by (i),(ii), and (iii), and by switching $m'$ and $i_{m'}$, we have that
\[
  x_{i_m} y_{m} z_{i_{m} m } \mid y_m^2 z_{i_m m} \ \ \text{ and }\ \ x_{m} z_{m l_{m}} y_{l_m} \mid x_m^2 z_{m l_m}
\]
After cancelations,
\[
  x_{i_m} \mid y_m \ \ \text{ and }\ \ y_{l_m} \mid x_m.
\]
Thus, $x_{i_m} = y_{l_m} = 1$. Since $i$ and $l$ are permutations, we have that $x_m = y_m$ for all $m$. Thus, our matrix looks like this
  \begin{eqnarray*}
    A = \left(\begin{array}{ccccc}
       z_{1 2} &  z_{2 1} &  z_{3 4} &  z_{4 3} & \cdots\\
       z_{2 1} &  z_{1 2} &  z_{4 3} &  z_{3 4} & \cdots
    \end{array}\right) . &  & 
  \end{eqnarray*}
But then, we have that, for example, $X_1Y_2 - X_2Y_1 = z_{1 2}^2 - z_{2 1}^2$, which is not prime! Contradiction! Thus, $z_{i j}=1$ for all $i\neq j$.

  Thus, $z_{1 2} = 1$ and by symmetry,
  $z_{i j} = 1$ for all $i \neq j$ and $\gcd ( X_i, Y_j) = 1$ for all $i \neq
  j$. Combined with the assumptions of the theorem statement, we get that
  $\gcd ( X_i, X_j) = \gcd ( X_i, Y_i) = \gcd ( Y_i, Y_j) = 1$ for all $i \neq
  j$, that is, all the entries of $A$ are pairwise relatively prime.

  ($\Leftarrow$): Assume that the entries of $A$ are pairwise relatively prime.  Let $<$ be any monomial term order.  To show that the $2\times 2$ minors of $A$ are a Gr\"obner basis with respect to $<$, we just need to show that all $S$-pairs reduce to zero.  By the result of Bernstin-Sturmfels-Zelevinsky, such a reduction is guaranteed to exist for a generic matrix $X = (x_{ij})$ .  Since all entries of $A$ are relatively prime, it is clear that such a reduction will extend simply by the ring map: $x_{ij} \mapsto X_{ij}.$
\end{proof}

\section{Robustness of Higher Betti Numbers}
Our interest in robust ideals originated with the following classical inequality: 
\begin{equation}\label{robust betti} \beta_{i}(S/\initial_< I)\leq \beta_{i} (S/I) \ \ \mbox{ for all } i.\end{equation}
It is natural to ask for which ideals and term orders equality holds (for all $i$).  In the setting of determinantal ideals, Conca et al proved in \cite{ConcaHostenEtal} that the ideal of maximal minors of a generic matrix has some initial ideal with this property.  They also gave examples of determinantal ideals for which no initial ideal has this property. In a different vein, Conca, Herzog and Hibi showed in \cite{ConcaHerzogHibi} that if the generic initial ideal ${\mathrm Gin}(I)$ has $\beta_i(I) = \beta_i ({\mathrm Gin}(I))$ for some $i>0$, then $\beta_k(I)$ and $\beta_k ({\mathrm Gin}(I))$ also agree for $k\geqslant i$.

Our interest was to instead approach the inequality \ref{robust betti} in a universal setting, i.e. to consider when equality holds for \emph{all} term orders $<$.  In this case we say that $I$ has \emph{robust Betti numbers}.  The following result is due to the first author \cite{Boocher}

\begin{thm}\label{all initial ideals have linear resolutions}
If $I:=I_k(X)$ is the ideal of maximal minors of a generic $k\times n$ matrix $X$ and $<$ is any term order, then 
\begin{equation}\label{robust betti} \beta_{ij}(S/\initial_< I)= \beta_{ij} (S/I) \ \ \mbox{ for all } i,j.\end{equation}
In particular, every initial ideal is a Cohen-Macaulay, squarefree monomial ideal with a linear free resolution.  Further, the resolution can be obtained from the Eagon-Northcott complex by taking appropriate lead terms of each syzygy.
\end{thm}
A combination of Theorems \ref{all initial ideals have linear resolutions} and \ref{mainthm} yields 

\begin{cor}
Let $I$ be a toric ideal generated in degree two.  If $I$ is robust, then $I$ has robust Betti numbers. 
\end{cor}

Our original hope with this project was that all robust toric ideals had robust Betti numbers.  Unfortunately, the situation seems much more delicate.  

\begin{example}
Using Gfan \cite{gfan}, we were able to check that the Lawrence ideal $J_L$ corresponding to the lattice $L$ given by the matrix 
$$\bgroup\begin{pmatrix}1&
        1&
        1&
        1&
        1\\
        0&
        1&
        {2}&
        {7}&
        {8}\\
        \end{pmatrix}\egroup$$
has initial ideals with different Betti numbers.
\end{example}  

\section{Acknowledgments}
Many of the results in this paper were discovered via computations using Macaulay 2 \cite{M2} and Gfan \cite{gfan}.  We thank David Eisenbud and Bernd Sturmfels for many useful discussions.  Finally we thank Seth Sullivant for introducing us to Lawrence ideals and starting us on the path toward Example \ref{counterexample}.

\bibliographystyle{plain}
\bibliography{biblio.bbl}

\begin{bibdiv}
\begin{biblist}

\bib{Boocher}{article}{
      author={Boocher, Adam},
       title={Free resolutions and sparse determinantal ideals},
        date={2012},
     journal={Math. Res. Letters},
      volume={19},
      number={4},
       pages={805\ndash 821},
}

\bib{MR1229427}{article}{
      author={Bernstein, David},
      author={Zelevinsky, Andrei},
       title={Combinatorics of maximal minors},
        date={1993},
        ISSN={0925-9899},
     journal={J. Algebraic Combin.},
      volume={2},
      number={2},
       pages={111\ndash 121},
         url={http://dx.doi.org/10.1023/A:1022492222930},
      review={\MR{1229427 (94j:52021)}},
}

\bib{ConcaGorlaDeNegri}{article}{
      author={Conca, Aldo},
      author={De~Negri, Emanuela},
      author={Gorla, Elisa},
       title={Universal groebner bases for maximal minors},
        date={2013},
     journal={ArXiv},
}

\bib{ConcaHerzogHibi}{article}{
      author={Conca, Aldo},
      author={Herzog, J{\"u}rgen},
      author={Hibi, Takayuki},
       title={Rigid resolutions and big {B}etti numbers},
        date={2004},
        ISSN={0010-2571},
     journal={Comment. Math. Helv.},
      volume={79},
      number={4},
       pages={826\ndash 839},
         url={http://dx.doi.org/10.1007/s00014-004-0812-2},
      review={\MR{2099124 (2005k:13025)}},
}

\bib{ConcaHostenEtal}{incollection}{
      author={Conca, Aldo},
      author={Ho{\c{s}}ten, Serkan},
      author={Thomas, Rekha~R.},
       title={Nice initial complexes of some classical ideals},
        date={2006},
   booktitle={Algebraic and geometric combinatorics},
      series={Contemp. Math.},
      volume={423},
   publisher={Amer. Math. Soc.},
     address={Providence, RI},
       pages={11\ndash 42},
         url={http://dx.doi.org/10.1090/conm/423/08073},
      review={\MR{2298753 (2008h:13035)}},
}

\bib{M2}{misc}{
      author={Grayson, Daniel~R.},
       title={Macaulay 2, a software system for research in algebraic
  geometry},
        note={Available at \url{http://www.math.uiuc.edu/Macaulay2/}},
}

\bib{gfan}{misc}{
      author={Jensen, Anders~N.},
       title={{G}fan, a software system for {G}r{\"o}bner fans and tropical
  varieties},
         how={Available at
  \url{http://home.imf.au.dk/jensen/software/gfan/gfan.html}},
}

\bib{ConvexPoly}{book}{
      author={Sturmfels, Bernd},
       title={Gr\"obner bases and convex polytopes},
      series={University Lecture Series},
   publisher={American Mathematical Society},
     address={Providence, RI},
        date={1996},
      volume={8},
        ISBN={0-8218-0487-1},
      review={\MR{1363949 (97b:13034)}},
}

\bib{MR1212627}{article}{
      author={Sturmfels, Bernd},
      author={Zelevinsky, Andrei},
       title={Maximal minors and their leading terms},
        date={1993},
        ISSN={0001-8708},
     journal={Adv. Math.},
      volume={98},
      number={1},
       pages={65\ndash 112},
         url={http://dx.doi.org/10.1006/aima.1993.1013},
      review={\MR{1212627 (94h:52020)}},
}

\end{biblist}
\end{bibdiv}
\end{document}